\documentclass[12pt, a4paper,leqno]{amsart}
\usepackage[T1]{fontenc}

\usepackage{amsmath}
\usepackage{amsfonts}
\usepackage{amssymb}
\usepackage{amsthm}

\usepackage{comment}
\usepackage{url}
 
\newcommand{\R}{\mathbb{R}}

\newcommand{\dif}[0]{\ensuremath{\,\mathrm{d}}}
\newcommand{\norm}[1]{\ensuremath{\left\Vert #1 \right\Vert}}
\newcommand{\abs}[1]{\ensuremath{\left\vert #1 \right\vert}}

\DeclareMathOperator*{\esssup}{ess\,sup}

\DeclareMathOperator*{\dive}{div}

\def\XXint#1#2#3{{\setbox0=\hbox{$#1{#2#3}{\int}$}
     \vcenter{\hbox{$#2#3$}}\kern-.5\wd0}}

\theoremstyle{plain}
\newtheorem{thm}{Theorem}
\newtheorem{lem}[thm]{Lemma}

\numberwithin{thm}{section}
\numberwithin{equation}{section}

\theoremstyle{definition}
\newtheorem{defn}[thm]{Definition}

\begin{document}

\title[A curious equation]{A curious equation involving the
  $\infty$-Laplacian}

\author{Peter Lindqvist} 

\author{Teemu Lukkari}


\date{\today}

\begin{abstract}
  We prove the uniqueness for viscosity solutions of the differential
  equation
  \begin{displaymath}
    \sum u_{x_i}u_{x_j}u_{x_i x_j}+\abs{\nabla u}^2\ln\abs{\nabla
      u}\langle\nabla u, \nabla \ln p \rangle=0.
  \end{displaymath}
  A variant of the Harnack inequality is derived.  The equation comes
  from the problem of finding
  \begin{displaymath}
    \min_u\max_x(\abs{\nabla u(x)}^{p(x)}).
  \end{displaymath}
  The positive exponent $p(x)$ is a continuously differentiable
  function.
\end{abstract}

\maketitle

\section{Introduction}

The object of our study is the differential equation 
\begin{equation}
  \label{eq:curious}
  \sum_{i,j=1}^{n} u_{x_i}u_{x_j}u_{x_i x_j}+\abs{\nabla u}^2\ln\abs{\nabla
      u}\langle\nabla u, \nabla \ln p \rangle=0.
\end{equation}
in a bounded domain $\Omega$ in $\R^n$. The main result is the
uniqueness for viscosity solutions, Theorem \ref{thm:uniqueness}.  We
also have a Harnack inequality in Section \ref{sec:harnack}.

In the case of the constant function $p(x)=p$, the last term is not
present, and the equation is well-known. During the last fifteen years
there has been a lot of research done for the $\infty$-Laplace
equation
\begin{equation}
  \label{eq:infty-lap}
  \Delta_\infty u\equiv\sum_{i,j=1}^{n} u_{x_i} u_{x_j}u_{x_i x_j}=0.
\end{equation}
This is the limit, as $p\to\infty$, of the Euler--Lagrange equations
\begin{equation}
  \label{eq:p-lap}
  \Delta_p u\equiv \abs{\nabla u}^{p-2}\Delta u+(p-2)\abs{\nabla
    u}^{p-4}\Delta_\infty u = 0
\end{equation}
for the variational integrals 
\begin{displaymath}
  \int_\Omega\abs{\nabla u(x)}^p\dif x.
\end{displaymath}
Thus it appears that the equation $\Delta_\infty u=0$ is formally the
Euler--Lagrange equation of the "variational integral"
\begin{displaymath}
  I(u)=\norm{\nabla u}_{L^\infty(\Omega)}=\esssup_{\Omega}\abs{\nabla u}.
\end{displaymath}
It is sometimes called Aronsson's Euler equation, after its
discoverer, who derived the equation in order to find the best
Lipschitz extension of given boundary values, cf. \cite{Aronsson}. The
equation must be interpreted in the sense of viscosity solutions. We
assume that the reader is familiar with the basic facts of this
fascinating theory, cf. \cite{UsersGuide, Primer, BeginnersGuide}.

In his remarkable work \cite{Jensen} R. Jensen succeeded in proving
the uniqueness of the viscosity solutions to \eqref{eq:infty-lap}. We
will closely follow Jensen's construction of auxiliary equations, when
we come to our uniqueness proof below.

Let us return to the equation \eqref{eq:curious}. The problem about 
\begin{displaymath}
  \min_u\max_x(\abs{\nabla u(x)}^{p(x)})
\end{displaymath}
with a variable exponent can be reached via the variational integrals
\begin{equation}\label{eq:kpx-integrals}
  \left\{\int_\Omega\abs{\nabla u(x)}^{kp(x)}
    \frac{\dif x}{kp(x)}\right\}^{\frac{1}{k}}
\end{equation}
as $k\to\infty$. Such integrals were first considered by Zhikov, cf.
\cite{Zhikov}. The Euler--Lagrange equation is
\begin{align*}
  \Delta_{kp(x)}u\equiv & \abs{\nabla u}^{kp(x)-2}\Delta u
  +(kp(x)-2)\abs{\nabla u}^{kp(x)-4}\Delta_\infty u
  \\&+\abs{\nabla u}^{kp(x)-2}\ln\abs{\nabla u}
  \langle\nabla u,\nabla k p(x)\rangle=0.
\end{align*}
Notice the extra term with $\nabla p$. Its formal limit as
$k\to\infty$ is the equation
\begin{equation}\label{eq:infty-x-lap}
  \Delta_{\infty(x)}u \equiv \Delta_\infty u+\abs{\nabla
    u}^2\ln\abs{\nabla u}\langle\nabla u,\nabla \ln p\rangle=0,
\end{equation}
and the reader can recognize \eqref{eq:curious}.

Thus the operator $\Delta_\infty$ is replaced by the new operator
$\Delta_{\infty(x)}$, the $\infty$-~Laplacian with variable exponent.
Again, the interpretation is delicate, since now $\nabla p$ is needed
pointwise. To be on the safe side, we therefore assume that $p\in
C^1(\Omega)\cap W^{1,\infty} (\Omega)$, $p(x)>1$, and that $\Omega$ is a
bounded domain in $\R^n$. Then viscosity solutions to
\eqref{eq:curious} can be defined in the standard way.

\begin{defn}
  We say that a lower semicontinuous function $v:\Omega\to
  (-\infty,\infty]$ is a \emph{viscosity supersolution} if, whenever
  $x_0\in \Omega$ and $\varphi\in C^2(\Omega)$ are such that
  \begin{enumerate}
  \item $\varphi(x_0)=u(x_0)$, and
  \item $\varphi(x)<v(x)$, when $x\not=x_0$, 
  \end{enumerate}
  we have
  \begin{displaymath}
    \Delta_{\infty(x_0)}\varphi(x_0)\leq 0.
  \end{displaymath}
\end{defn}

The \emph{viscosity subsolutions} have a similar definition; they are
upper semicontinuous, the test functions touch from above and the
differential inequality is reversed. Finally, a \emph{viscosity
  solution} is a function that is both a viscosity supersolution and
viscosity subsolution. A very simple example is
\begin{displaymath}
  u(x)=\abs{x};
\end{displaymath}
it is a viscosity subsolution for all exponents $p(x)$.

The boundary values are prescribed by a Lipschitz continuous function
$f:\partial\Omega\to\R$. It can be extended to the whole space with
the same constant, say
\begin{displaymath}
  \abs{f(x)-f(y)}\leq L\abs{x-y}, \quad x,y\in \R^n.
\end{displaymath}
According to Rademacher's theorem $f$ is differentiable a.e. and
$\norm{\nabla f}_\infty\leq L$. After the extension, one has $f\in
W^{1,\infty}(\R^n)$.

The following existence and uniqueness result holds for the Dirichlet
boundary value problem. 

\begin{thm}\label{thm:uniqueness}
  Given a Lipschitz continuous function $f:\partial\Omega\to\R$, there
  exists a unique viscosity solution $u\in C(\overline{\Omega})$ with
  boundary values $f$. Moreover, $u\in W^{1,\infty}(\Omega)$, and
  $\norm{\nabla u}_{L^\infty(\Omega)}$ has a bound depending only on
  the Lipschitz constant of $f$.
\end{thm}

The main part of the proof is in Lemma \ref{lem:comparison} and Lemma
\ref{lem:variational-comparison}. The result below can be extracted
from our constructions.

\begin{thm}[Comparison principle]
  If $u$ is a viscosity subsolution and $v$ a viscosity supersolution
  to \eqref{eq:curious} in $\Omega$, then $v\geq u$ on $\partial
  \Omega$ implies that $v\geq u$ in $\Omega$.
\end{thm}

Passing a Caccioppoli estimate for the minimizers of the integrals
\eqref{eq:kpx-integrals} to the limit, we obtain the following form of
Harnack's inequality. This is similar to the one proved by Alkhutov
\cite{Alkhutov}.

\begin{thm}[Harnack's inequality]\label{thm:harnack}
  Let $u$ be a nonnegative viscosity solution of \eqref{eq:curious}.
  Then the inequality
  \begin{displaymath}
    \sup_{x\in B_R}u(x)\leq C\left(\inf_{x\in B_R}u(x)+R\right)
  \end{displaymath}
  holds, with the constant depending on the supremum of $u$ taken over
  $B_{2R}\subset\Omega$.
\end{thm}

\section{The auxiliary equations}
\label{sec:auxiliary}

Following a device of Jensen in \cite{Jensen}, we introduce two
auxiliary equations with a positive parameter $\varepsilon$. The
situation will be
\begin{align}
  &\max\{\varepsilon-\abs{\nabla v},\Delta_{\infty(x)}v\}=0 
  & &\text{\textit{Upper equation}}\label{eq:upper}\\
  &\Delta_{\infty(x)}h=0 &  &\text{\textit{Equation}} \label{eq:equation}\\ 
  &\min\{\abs{\nabla u}-\varepsilon,\Delta_{\infty(x)}u\}=0   
  & &\text{\textit{Lower equation}}\label{eq:lower}
\end{align}
If the solutions have the same boundary values, it turns out that
$u\leq h\leq v$. We say that a continuous function $v:\Omega\to \R$ is
a viscosity supersolution of the upper equation, if whenever $x_0\in
\Omega$ and $\varphi \in C^2(\Omega)$ are such that
\begin{enumerate}
\item $\varphi(x_0)=v(x_0)$,
\item $\varphi(x)<v(x)$, when $x\not=x_0$
\end{enumerate}
then we have 
\begin{displaymath}
  \varepsilon-\abs{\nabla\varphi(x_0)}\leq 0\text{ and
  }\Delta_{\infty(x_0)}\varphi(x_0)\leq 0.
\end{displaymath}
Notice that the differential operator is evaluated only at the
touching point $x_0$. The other definitions are analogous; the
viscosity subsolutions of the lower equation are also used later.

The existence of solutions is proved through the following variational
procedure. The equation 
\begin{displaymath}
  \Delta_{kp(x)}u=-\varepsilon^{kp(x)-1},
\end{displaymath}
or in its weak form
\begin{equation}
  \label{eq:upper-weak}
  \int_{\Omega}\langle \abs{\nabla u}^{kp(x)-2}\nabla u,\nabla\eta\rangle\dif x
  =\int_\Omega \varepsilon^{kp(x)-1}\eta\dif x,
\end{equation}
when $\eta\in C_0^\infty(\Omega)$, is the Euler--Lagrange equation of
the variational integral
\begin{equation}
  \label{eq:upper-variational}
  \int_{\Omega}\abs{\nabla u}^{kp(x)}\frac{\dif x}{kp(x)}
  -\int_{\Omega}\varepsilon^{kp(x)-1}\dif x.
\end{equation}
As $k\to\infty$, we get the \emph{upper equation}
\begin{equation}
  \label{eq:upper-limit}
  \max\{\varepsilon-\abs{\nabla v},\Delta_{\infty(x)}v\}=0.
\end{equation}
Suppose now that $u_k\in C(\overline{\Omega})\cap
W^{1,kp(x)}(\Omega)$, $u_k=f$ on $\partial\Omega$, is the minimizer of
the above variational integral. Then it satisfies the weak equation
and a standard procedure shows that it is a viscosity solution, see
for example \cite{Jensen} about the method.  It follows from the
minimizing property that
\begin{displaymath}
  \int_{\Omega}\abs{\nabla v_k}^{kp(x)}\frac{\dif x}{kp(x)}
  \leq \int_{\Omega} \abs{\nabla f}^{kp(x)}\frac{\dif
    x}{kp(x)}
  +\int_{\Omega} \varepsilon^{kp(x)-1}(v_k-f)\dif x,
\end{displaymath}
since $f$ is admissible.We can see that $v_k\to v\in C(\overline{\Omega}) \cap
W^{1,\infty}(\Omega)$ uniformly in $\Omega$, at least for a
subsequence. We have
\begin{displaymath}
  \norm{\abs{\nabla v}^{p(x)}}_{\infty}\leq \norm{\abs{\nabla
      f}^{p(x)}}_{\infty} 
  +\norm{\varepsilon^{p(x)}}_{\infty}.
\end{displaymath}
It is again a standard procedure to verify that this $v$ is a
viscosity supersolution of the upper equation, cf. \cite{Jensen}. (It
is also a viscosity subsolution.)  We call $v$ a \emph{variational
  solution}, because it is a limit of minimizers of variational
integrals. We record an immediate estimate.
\begin{lem}\label{lem:estimate}
  A variational solution of the upper equation satisfies
  \begin{displaymath}
    \norm{\nabla v}_{\infty}\leq K,
  \end{displaymath}
  where $K$ depends only on the Lipschitz constant of the boundary
  values.
\end{lem}

For the lower auxiliary equation 
\begin{equation}
  \label{eq:lower-limit}
  \min\{\abs{\nabla u}-\varepsilon,\Delta_{\infty(x)}u\}=0
\end{equation}
the various stages are 
\begin{displaymath}
  \Delta_{kp(x)}u=\varepsilon^{kp(x)-1},
\end{displaymath}
\begin{displaymath}
  \int_{\Omega}\langle\abs{\nabla u}^{kp(x)-2}\nabla u,\nabla\eta\rangle\dif x
  =-\int_\Omega \varepsilon^{kp(x)-1}\eta\dif x,  
\end{displaymath}
\begin{displaymath}
  \int_{\Omega}\abs{\nabla u}^{kp(x)}\frac{\dif x}{kp(x)}
  +\int_{\Omega}\varepsilon^{kp(x)-1}\dif x.
\end{displaymath}
The situation is analogous to the previous case, but now the
subsolutions count.

We have to construct solutions of the auxiliary equations that are
close for small values of $\varepsilon$. Let $u^-_k$, $u_k$, and
$u_k^+$ be the weak solutions of the lower equation \eqref{eq:lower},
the equation \eqref{eq:equation} and the upper equation
\eqref{eq:upper}, all with the same boundary values $f$.  Then
\begin{displaymath}
  u^-_k\leq u_k \leq u_k^+
\end{displaymath}
by comparison. The weak solutions are viscosity solutions of their
respective equations.  Select a subsequence of indices so that all
three converge, say $u_k^-\to u^-$, $u_k\to h$, and $u^+_k\to u^+$.
Thus
\begin{align*}
  \dive(\abs{\nabla u^+_k}^{kp(x)-2}\nabla
  u_k^+)&=-\varepsilon^{kp(x)-1}\\
  \dive(\abs{\nabla u^-_k}^{kp(x)-2}\nabla
  u_k^-)&=+\varepsilon^{kp(x)-1},
\end{align*}
and, using $u^+_k-u^-_k$ as a test function in the weak formulation of
the equations and subtracting these, we obtain
\begin{align*}
  \int_{\Omega}\langle \abs{\nabla u^+_k}^{kp(x)-2}\nabla u_k^+
  -&\abs{\nabla u^-_k}^{kp(x)-2}\nabla u_k^-,
  \nabla u_k^+-\nabla u_k^- \rangle\dif x\\
  =&\int_{\Omega}\varepsilon^{kp(x)-2}(u_k^+-u^-_k)\dif x.
\end{align*}
With the aid of the elementary inequality
\begin{displaymath}
  \langle\abs{b}^{q-2}b-\abs{a}^{q-2}a,b-a\rangle\geq 2^{2-q}\abs{b-a}^q
\end{displaymath}
for vectors $b,a$ and $q\geq 2$, we obtain
\begin{displaymath}
  4\int_{\Omega}\abs{\frac{\nabla u^+_k-\nabla
      u_k^-}{2}}^{kp(x)}\dif x
  \leq
  \frac{1}{\varepsilon}\int_{\Omega}\varepsilon^{kp(x)}(u_k^+-u_k^-)\dif x.
\end{displaymath}
Extracting the $k$th roots, we conclude that
\begin{displaymath}
  \esssup\abs{\frac{\nabla u^+-\nabla u^-}{2}}^{p(x)}
  \leq \sup(\varepsilon^{p(x)}).
\end{displaymath}
Keeping $\varepsilon<1$, we arrive at the estimates
\begin{align*}
  \norm{\nabla u^+-\nabla u^-}_{\infty}\leq & C'\varepsilon^{\kappa}\\
  \norm{u^+-u^-}_{\infty}\leq C\varepsilon^{\kappa}
\end{align*}
where $\kappa$ depends only on the bounds on $p(x)$. The obtained
functions $u^+$, $u^-$ and $h$ are viscosity solutions of their
equations.

We have the result
\begin{displaymath}
  u^-\leq h\leq u^+\leq u^-+\mathcal{O}(\varepsilon^\kappa)
\end{displaymath}
for solutions $u^-$ (lower equation), $h$ (the equation), and $u^+$
(upper equation) coming from the variational procedure. This does not
yet prove that variational solutions are unique. The possibility that
another subsequence yields three new ordered solutions is difficult to
exclude. (At least it can be arranged so that the same $h$ will do for
all $\varepsilon$, though the constructed $u^-$, $u^+$ depend on
$\varepsilon$.)

\begin{lem}\label{lem:variational-comparison}
  If $u\in C(\overline{\Omega})$ is an arbitrary viscosity solution of the
  equation, $u=f$ on $\partial \Omega$, then
  \begin{displaymath}
    u^-\leq u\leq u^+,
  \end{displaymath}
  where $u^-$, $u^+$ are the constructed variational solutions of the
  auxiliary equations. 
\end{lem}
This lemma, which will be proved in Section \ref{sec:comparison},
implies that
\begin{displaymath}
  \abs{h-u}\leq \mathcal{O}(\varepsilon^\kappa).
\end{displaymath}
Since $u$ is independent of $\varepsilon$, it is unique; use 
\begin{displaymath}
  \abs{u_1-u_2}\leq \abs{h-u_1}+\abs{h-u_2}\leq\mathcal{O}(\varepsilon^\kappa) 
\end{displaymath}
and let $\varepsilon\to 0$ to see that two viscosity solutions $u_1$,
$u_2$ coincide. Thus Theorem~\ref{thm:uniqueness} follows. This also
shows that $u$ is a variational solution, since $u=h$.

\section{Proof of the comparison principle}

\label{sec:comparison}

Recall the variational solutions $u^+$ and $u^-$ of the auxiliary
equations in Section \ref{sec:auxiliary}. They satisfy the
inequalities
\begin{displaymath}
  \varepsilon\leq \abs{\nabla u^{\pm}}\leq K.
\end{displaymath}
The crucial part of the proof is in the following lemma.
\begin{lem}\label{lem:comparison}
  If $u\in C(\overline{\Omega})$ is a viscosity subsolution of the
  equation \eqref{eq:curious}, and if $u\leq f = u^+$ on
  $\partial\Omega$, then $u\leq u^+$ in $\Omega$.
  
  The analogous comparison holds for viscosity supersolutions lying
  above $u^-$.
\end{lem}
\begin{proof}
  By adding a constant we may assume that $u^+>0$. Write $v=u^+$ for
  simplicity. We claim that $v\geq u$. We use the \emph{antithesis}
  \begin{displaymath}
    \max_{\Omega}(u-v)>\max_{\partial \Omega}(u-v).
  \end{displaymath}
  We will construct a strict supersolution $w=g(v)$ of the upper
  equation such that also
  \begin{equation}\label{eq:comparison-proof1}
    \max_{\Omega}(u-w)>\max_{\partial \Omega}(u-w),
  \end{equation}
  and 
  \begin{displaymath}
    \Delta_{\infty(x)}w\leq -\mu<0
  \end{displaymath}
  in $\Omega$. This will lead to a contradiction.

  In fact, we will use the expedient approximation 
  \begin{displaymath}
    g(t)=\frac{1}{\alpha}\log(1+A(e^{\alpha
      t}-1))
  \end{displaymath}
  of the identity, which was studied in
  \cite{JuutinenLindqvistManfrediARMA}. Here $A>1$ and $\alpha>0$. Now
  \begin{align*}
    0&<g(t)-t<\frac{A-1}{\alpha},\\
    0&<g'(t)-1<A-1,\\
  \end{align*}
  assuming that $t\geq 0$. Further
  \begin{equation}\label{eq:g1}
    \frac{g''(t)}{g'(t)}=-\alpha[g'(t)-1]=
    -\frac{\alpha(A-1)}{1+A(e^{\alpha t}-1)},
  \end{equation}
  and
  \begin{equation}\label{eq:g2}
    0\leq \log(g'(t))=\log[1+(g'(t)-1)]\leq g'(t)-1,
  \end{equation}
  provided that $A<2$.

  To prove the comparison, we need the equation for $w=f(v)$. We have
  \begin{align*}
    w          &=g(v), \quad w_{x_i}=g'(v)v_{x_i},\\
    w_{x_i x_j}&=g''(v)v_{x_i}v_{x_j}+g'(v)v_{x_i x_j}\\
    \Delta_\infty w&=g'(v)^3\Delta_\infty v +g'(v)^2g''(v)\abs{\nabla v}^4.
  \end{align*}
  Multiplying the upper equation (for supersolutions)
  \begin{displaymath}
    \max\{\varepsilon-\abs{\nabla v},\Delta_{\infty(x)}v\}\leq 0
  \end{displaymath}
  by $g'(v)^3$, we formally obtain that
  \begin{displaymath}
    \Delta_{\infty}w-g'(v)^2 g''(v)\abs{\nabla v}^4
    +\abs{\nabla w}^2\ln(\abs{\nabla v})
    \langle \nabla w,\nabla \ln p\rangle\leq 0.
  \end{displaymath}
  Writing $\ln\abs{\nabla v}=\ln(\abs{\nabla w})-\ln(g'(v))$ we obtain
  the equation
  \begin{multline*}
    \Delta_\infty w+\abs{\nabla w}^2\ln(\abs{\nabla w})
    \langle\nabla w,\nabla \ln p\rangle\\
    \leq \abs{\nabla w}^3\left[\frac{g''(v)}{g'(v)}\abs{\nabla v}
    +\ln(g'(v))\abs{\nabla \ln p}\right].
  \end{multline*}
  We also have $\varepsilon\leq \abs{\nabla v}$. Using \eqref{eq:g1}
  and \eqref{eq:g2}, we can write
  \begin{align*}
    \Delta_\infty w\leq &\abs{\nabla w}^3
    \left[-\alpha\varepsilon+\norm{\nabla \ln p}_\infty\right]
    \frac{A-1}{1+A(e^{\alpha v}-1)}\\
    \leq & \,\varepsilon^3 g'(v)^4[-\alpha\varepsilon+\norm{\nabla \ln
      p}_\infty](A-1)A^{-1}e^{-\alpha v}.
  \end{align*}
  Given $\varepsilon>0$, fix $\alpha=\alpha(\varepsilon)$ so large
  that
  \begin{displaymath}
    -\alpha\varepsilon+\norm{\nabla\ln p}_\infty\leq -2.
  \end{displaymath}
  Then fix $A$ so close to 1 that
  \begin{displaymath}
    0<w-v=g(v)-v<\frac{A-1}{\alpha} <\delta,
  \end{displaymath}
  where $\delta$ is small enough to guarantee
  \eqref{eq:comparison-proof1}. With these adjustments, the right-hand
  member is less than the \emph{negative} quantity
  \begin{displaymath}
    -\mu=-\varepsilon^3 1^4 (A-1)A^{-1}e^{-\alpha\norm{v}_\infty}.
  \end{displaymath}
  The resulting equation is 
  \begin{displaymath}
    \Delta_{\infty(x)}w\leq -\mu.
  \end{displaymath}
  The described procedure was formal. The reader should replace $v$ by
  a test function $\varphi$ touching $v$ from below at a point $x_0$
  and $w$ should be replaced by the test function $\psi=g(\varphi)$,
  which touches $w$ from below. The inversion $\varphi=g^{-1}(\psi)$
  is evident. We have proved that
  \begin{displaymath}
    \Delta_{\infty(x_0)}\psi(x_0)\leq -\mu
  \end{displaymath}
  whenever $\psi$ touches $w$ from below at $x_0\in\Omega$. 

  We aim at using "the theorem of sums", formulated in terms of the
  so-called superjets and subjets.  For the jets and their closures,
  we refer to \cite{UsersGuide, Primer, BeginnersGuide}. Start with
  doubling the variables:
  \begin{displaymath}
    M=\sup_{\substack{x\in\Omega\\y\in\Omega}}\left(u(x)-w(y)
      -\frac{j}{2}\abs{x-y}^2\right).
  \end{displaymath}
  The maximum is attained at the interior points $x_j$, $y_j$ (for
  large indices) and 
  \begin{displaymath}
    x_j\to \Hat{x}, \qquad y_j\to \Hat{x},
  \end{displaymath}
  where $\Hat{x}$ is an interior point, the same for both sequences.
  It cannot be on the boundary due to \eqref{eq:comparison-proof1}. (It
  is known that $j\abs{x_j-y_j}^2\to 0$.) We need the bound
  \begin{displaymath}
    j\abs{x_j-y_j}\leq C.
  \end{displaymath}
  To obtain it, we reason as follows:
  \begin{multline*}
    u(x_j)-w(y_j)-\frac{j}{2}\abs{x_j-y_j}^2\\ \geq 
    u(x_j)-w(x_j)-\frac{j}{2}\abs{x_j-x_j}^2=u(x_j)-w(x_j),
  \end{multline*}
  so that
  \begin{displaymath}
    \frac{j}{2}\abs{x_j-y_j}^2\leq w(x_j)-w(y_j)
    \leq \norm{\nabla w}_{\infty}\abs{x_j-y_j}.
  \end{displaymath}
  Now $\norm{\nabla w}_{\infty}= \norm{g'(v)\nabla v}_{\infty}\leq
  KA$ according to Lemma \ref{lem:estimate}. The bound follows
  with $C=2KA$. (This is why $v$ has to be a variational
  solution!)  Further, we need the bound 
  \begin{displaymath}
    j\abs{x_j-y_j}\geq\varepsilon,
  \end{displaymath}
  which follows from $\nabla w=g'(v)\nabla v$, since $g'(v)\geq 1$ and
  $\abs{\nabla v}\geq \varepsilon$ in the viscosity sense. 

  The theorem of sums assures that there exist symmetric $n\times
  n$-matrices $X_j$ and $Y_j$ such that $X_j\leq Y_j$ and
  \begin{align*}
    (j(x_j-y_j),X_j)\in & \overline{J^{2,+}}u(x_j),\\
    (j(x_j-y_j),Y_j)\in & \overline{J^{2,-}}w(y_j)
  \end{align*}
  where $\overline{J^{2,+}}u(x_j)$ and $\overline{J^{2,-}}w(y_j)$ are
  the closures of the super- and subjets, respectively. We can rewrite
  the equations as
  \begin{align*}
    j^2&\langle Y_j(x_j-y_j),x_j-y_j\rangle \\
    &+ j^3\abs{x_j-y_j}^2 \ln(j\abs{x_j-y_j})\langle x_j-y_j,
    \nabla \ln p(y_j)\rangle \leq -\mu,\\
     j^2&\langle X_j(x_j-y_j),x_j-y_j\rangle \\
    &+ j^3\abs{x_j-y_j}^2 \ln(j\abs{x_j-y_j})\langle x_j-y_j,
    \nabla \ln p(x_j)\rangle \geq 0,\\
    j&\abs{x_j-y_j}\geq \varepsilon,\\
    j&\abs{x_j-y_j}\leq C.
  \end{align*}
  Subtract the equations and move the terms with $\ln p$. Then
  \begin{multline*}
    j^2\langle (Y_j-X_j)(x_j-y_j),(x_j-y_j)\rangle \\
    \begin{aligned}
      \leq  & -\mu\\ & +j^3\abs{x_j-y_j}^2\ln(j\abs{x_j-y_j})
      \langle x_j-y_j,\nabla \ln
      p(x_j)-\nabla\ln p(y_j)\rangle\\
      \leq &-\mu+ C^3\ln\left(\frac{C}{\varepsilon}\right)
      \abs{\nabla \ln p(x_j)-\nabla\ln p(y_j)}.
    \end{aligned}
  \end{multline*}
  The last term approaches zero as $j\to\infty$, because of the
  continuity. The very first term is non-negative, since $Y_j\geq
  X_j$. The contradiction
  \begin{displaymath}
    0\leq -\mu +0
  \end{displaymath}
  arises. Therefore, the antithesis is false, and consequently $u\leq
  v$. This concludes the proof.
\end{proof}

\section{Estimates for solutions}
\label{sec:harnack}

In this section, we prove some simple estimates for the positive
viscosity solutions of \eqref{eq:curious}.  In particular, we show
that they satisfy a version of Harnack's inequality, similar to the one
for solutions of the $p(x)$-Laplacian, cf. \cite{Alkhutov}.

We start by deriving a Caccioppoli estimate for finite exponents. Let
$u$ be a nonnegative minimizer, and set
\begin{displaymath}
  v(x)=u(x)+\varepsilon\zeta(x)^{p(x)}u(x)^{1-p(x)},  
\end{displaymath}
where $\zeta\in C^\infty_0(\Omega)$ is nonnegative. Then
\begin{multline*}
  \nabla v= \left(1-(p-1)\varepsilon
    \left(\frac{\zeta }{u}\right)^{p}\right)\nabla u\\
  +\varepsilon(p-1)\left(\frac{\zeta }{u}\right)^{p}
  \left[\frac{p}{p-1}\left(\frac{u}{\zeta}\right)\nabla \zeta
  +\frac{1}{p-1}u\ln\left(\frac{\zeta}{u}\right)\nabla p\right].
\end{multline*}
Observe that $\nabla v$ is a convex combination of $\nabla u$ and the
expression in the brackets, provided that $\lambda =
(p-1)\varepsilon\left(\frac{\zeta}{u}\right)^{p}\leq 1$. To see that
we may choose $\varepsilon$ sufficiently small to accomplish this,
first consider $u+\delta>0$ instead of $u$. Since $\varepsilon$ will
disappear from the estimate, we may safely let $\delta\to 0$ in the
end.

We use the minimizing property of $u$ and the convexity to get
\begin{multline*}
  \int_{\Omega}\abs{\nabla u}^{p}\frac{\dif x}{p}\leq 
  \int_{\Omega}\abs{\nabla v}^{p} \frac{\dif x}{p}\\
  \begin{aligned}
    \leq & \int_{\Omega}\left(1-\varepsilon(p-1)
      \left(\frac{\zeta}{u}\right)^p\right)\abs{\nabla u}^{p}
    \frac{\dif x}{p}\\
    &+\int_{\Omega}\varepsilon(p-1)\left(\frac{\zeta}{u}\right)^p
    \abs{\frac{p}{p-1}\left(\frac{u}{\zeta}\right) \nabla \zeta
      +\frac{1}{p-1}u\ln\left(\frac{\zeta}{u}\right)
      \nabla p}^{p}\frac{\dif x}{p}.
  \end{aligned}
\end{multline*}
This simplifies to 
\begin{displaymath}
  \int_\Omega (p-1)\zeta^p\abs{\nabla \ln u}^p\frac{\dif x}{p}\leq 
  \int_{\Omega}\frac{p^p}{(p-1)^{p-1}}\abs{\nabla \zeta 
    +\zeta\ln\left(\frac{\zeta}{u}\right)\frac{\nabla
      p}{p}}^{p}\frac{\dif x}{p} 
\end{displaymath}
where $\varepsilon$ has now disappeared.  

Next, we replace $p(x)$ by $kp(x)$; the above inequality then holds
for the corresponding minimizers $u_k$, and we can assume that they
converge uniformly to the solution $u$ of \eqref{eq:curious}. Then we
can take $k$th roots and let $k\to \infty$. This gives the following
lemma.

\begin{lem} \label{lem:infinity-caccioppoli} Let $u$ be a positive
  viscosity solution to \eqref{eq:curious}, and $\zeta$ a positive,
  compactly supported smooth function. Then
  \begin{displaymath}
    \sup_{x\in \Omega}\abs{\zeta(x)\nabla\ln u(x)}^{p(x)}\leq 
    \sup_{x\in \Omega}\abs{\nabla\zeta(x)
      +\zeta(x)\ln\left(\frac{\zeta(x)}{u(x)}\right)\nabla \ln p(x)}^{p(x)}.
  \end{displaymath}
\end{lem}

Observe that Lemma \ref{lem:infinity-caccioppoli} reduces to a
well-known estimate for solutions to the $\infty$-~Laplace equation
\eqref{eq:infty-lap} when $p(x)$ is constant.

Harnack's inequality is now a rather simple consequence of Lemma
\ref{lem:infinity-caccioppoli}. Indeed, take a cutoff function $\zeta$
compactly supported in $B(x_0,2R)$ such that $\zeta=1$ in $B(x_0,R)$,
$0\leq \zeta\leq 1$ and $\abs{\nabla \zeta}\leq 2/R$. Then, for
$v=u+R$, we get by the fundamental theorem of calculus that
\begin{displaymath}
  \abs{\ln v(x)-\ln v(y)}\leq \left(1+
  \norm{\abs{\zeta\nabla\ln v}^{p(x)}}_{\infty,B(x_0,2R)}\right)\abs{x-y}
\end{displaymath}
for $x, y\in B(x_0,R)$. Observing that $\abs{\ln v}\leq
\max\{R^{-1},\norm{v}_{\infty,\,B(x_0,2R)}\}$, Lemma
\ref{lem:infinity-caccioppoli} implies that the right hand side is
estimated by
\begin{displaymath}
  C_1\norm{v}_\infty\abs{x-y}+C_2\frac{\abs{x-y}}{R}.
\end{displaymath}
Taking exponents of both sides and replacing $v$ by $u+R$ we get
\begin{displaymath}
  u(x)+R\leq \exp(C_2\abs{x-y}/R)\exp(C_1\norm{u+R}_\infty\abs{x-y})(u(y)+R).
\end{displaymath}
The Harnack inequality of Theorem \ref{thm:harnack} follows from this,
although the estimate is more powerful.

A couple of variants are also obtained by similar reasoning. For
instance, one can replace $R$ by $R^\alpha$ for any $\alpha> 0$,
the price to pay being that $C_2=\mathcal{O}(\alpha)$.  Similarly, one
can have any positive power $\varepsilon$ on the supremum of $u$ in
the constant, with $C_1=\mathcal{O}(1/\varepsilon)$.

\def\cprime{$'$}

\end{document}